\documentclass[12pt]{amsart}
\usepackage{latexsym, enumerate}
\usepackage{amsmath, amssymb, tikz, mathrsfs, inputenc, verbatim, bbm, tikz, pgffor, epsfig, mathrsfs, xcolor, mathtools, comment}

\usepackage{tikz}
\usetikzlibrary{shapes.geometric,arrows}

\tikzstyle{box} = [rectangle,text centered, draw=black]
\tikzstyle{arrow} = [thick, ->, >=stealth]

\usepackage{color}

\usepackage{amsthm}
\theoremstyle{plain}
\newtheorem*{theorem*}{Theorem}

\renewcommand{\epsilon}{\varepsilon}

\newcommand{\N}{{\mathbb N}}
\newcommand{\R}{{\mathbb R}}

\renewcommand{\phi}{\varphi}

\newcommand{\IO}{\mathcal{O}}

\numberwithin{equation}{section}

\newtheorem{theorem}{Theorem}[section]

\newtheorem{proposition}[theorem]{Proposition}
\newtheorem{corollary}[theorem]{Corollary}
\newtheorem{lemma}[theorem]{Lemma}

\theoremstyle{definition}

\theoremstyle{remark}

\DeclarePairedDelimiter\floor{\lfloor}{\rfloor}

\begin{document}
%\title{  Erd\H{o}s similarity Conjecture on a Cantor-type set}
\title{Large sets avoiding affine copies of infinite sequences}
\author{Angel Cruz}
\address{Angel Cruz, San Francisco State University, Department of Mathematics, 1600 Holloway Avenue, CA 94132, US}
\email{angelc@mail.sfsu.edu}
\author{Chun-Kit Lai}
\address{Chun-Kit Lai, San Francisco State University, Department of Mathematics, 1600 Holloway Avenue, CA 94132, US}
\email{cklai@sfsu.edu}
\author{Malabika Pramanik}
\address{Malabika Pramanik, University of British Columbia, Department of Mathematics, 1984 Mathematics Road, Vancouver, BC V6T 1Z2}
\email{malabika@math.ubc.ca}

\date{\today}

\maketitle
{\allowdisplaybreaks}

\begin{abstract}
A conjecture of Erd\H{o}s states that for any infinite set $A \subseteq \mathbb R$, there exists $E \subseteq \mathbb R$ of positive Lebesgue measure that does not contain any nontrivial affine copy of $A$. The conjecture remains open for most fast-decaying sequences, including the geometric sequence $A = \{2^{-k} : k \geq 1\}$. In this article, we consider infinite decreasing sequences $A = \{a_k: k \geq 1\}$  in $\R$ that converge to zero at a prescribed rate; namely $\log (a_n/a_{n+1}) = e^{\varphi(n)} $, where $\varphi(n)/n\to 0$ as $n\to\infty$. This condition is satisfied by sequences whose logarithm has polynomial decay, and in particular by the geometric sequence.  For any such sequence $A$, we construct a Borel set ${\mathcal O}\subseteq \mathbb R$ of Hausdorff dimension 1, but Lebesgue measure zero, that avoids all nontrivial affine copies of $A\cup\{0\}$.  
%that there exists a compact set $K \subseteq \mathbb R$ of positive Lebesgue measure such that the set of points that avoid all nontrivial affine copies of $A$ has Hasudorff dimension one. The condition is satisfied by sequences with exponential decay. 
\end{abstract}

\section{Introduction} 

Following the terminologies  in \cite{LP2009}, let us call a set $A\subseteq {\mathbb R}$ {\it universal} if every  set of positive Lebesgue measure in $\R$ contains a non-trivial affine copy of the set $A$. In other words, $A$ is universal if for every $E \subseteq \mathbb R$ with $m(E) > 0$, there exist $x\in{\mathbb R}$ and $\delta\ne 0$ such that
%By a non-trivial affine copy, we mean a translated and rescaled version of $A$, i.e. the set of the form
$x+\delta A \subseteq E$. 
A classical result of Steinhaus \cite{S1920} shows, using the Lebesgue density theorem, that  finite sets must be universal. In 1974, Erd\H{o}s proposed a conjecture, now known as the {\em{Erd\H{o}s similarity conjecture}} \cite{E1974}, \cite[Chapter 4]{E-Scottish}:

\medskip

\noindent{\bf Conjecture.} {\it There are no infinite universal sets. }

\medskip

We will call an  infinite sequence $A = \{a_k\}$ a {\it null sequence} if  $a_k>0$ and strictly decreases to $0$. It is easy to see that the conjecture will be resolved in full generality if all null sequences are shown to be non-universal. The conjecture is currently open for sequences with exponential decay, and in particular for $A = \{2^{-k}: k \geq 1\}$.  

\medskip

In this paper, we study, for a given a compact set $K \subseteq \mathbb R$, the following set:
\begin{equation}\label{Eset}
{\mathcal E} =\mathcal{E}_K \coloneqq \{x\in K: \forall \delta\ne 0, \exists \ k\in{\mathbb N} \ \mbox{s.t.} \ x+\delta a_k\not\in K\},
\end{equation}
which is the set of translates $x$ such that $x + \delta A\not\subseteq K$ for every $\delta \ne 0$.   We call $\mathcal E_K$ the {set of \it Erd\H{o}s points} of $K$. In the appendix, we will show that 
\begin{enumerate}[(a)]  \label{appendix-results}
\item A set $A$ is universal if and only if every {\em{compact}} set $K \subseteq \mathbb R$ with $m(K) > 0$ contains a nontrivial affine copy of $A$ (Lemma \ref{universal-compact-lemma}), and  
\item For every compact set $K \subseteq \mathbb R$, the set $\mathcal{E}_K$ is Borel measurable (Proposition \ref{E-Borel-prop}). 
\end{enumerate}
%First of all, as a preliminary reduction of the conjecture, it suffices to show that  null sequences are not universal. To see this, we note that universal sets $A$ must be bounded since if $A$ is unbounded, then no bounded sets can contain an affine copy of $A$. Therefore, we can extract a subsequence of $A$ with limit point. By a translation, we can always assume the limit point is $0$. By reflection and taking subsequences if necessary, we can also assume that all $a_k$ are positive and strictly decreasing. 

\medskip

Our main result is the following:
\begin{theorem} \label{mainthm}
Let  $A = \{a_k : k \geq 1\}$ be a null sequence. Suppose that 
\begin{equation}\label{moderate_decay_condition}
\log\left(\frac{a_n}{a_{n+1}}\right)  = e^{\varphi(n)}
\end{equation}
where $\varphi(n)$ is strictly increasing and $\frac{\varphi(n)}{n} \to 0 $ as $n\to\infty$. Then there exists a compact set $K$ of Lebesgue measure arbitrarily close to 1 such that the set of Erd\H{o}s points has Hausdorff dimension 1. 
\end{theorem} 

\medskip

The set $K$ mentioned in Theorem \ref{mainthm} will be of Cantor type, whose construction is described in Section \ref{section-Cantor}. By restricting the set $K$ to ${\mathcal E}$, we establish the following corollary.

\begin{corollary} \label{mainthm-2}
Let  $A = \{a_k: k \geq 1\}$ be a null sequence obeying (\ref{moderate_decay_condition}). Then there exists a Borel set $\mathcal E$ of Hausdorff dimension 1 that does not contain any nontrivial affine copy of $A\cup \{0\}$.
\end{corollary}

%\begin{rem}\label{remark} We have the following remark concerning our main results.

{\em{Remarks:}}
\begin{enumerate}[1.]

\smallskip

\item  The condition in (\ref{moderate_decay_condition}) is satisfied for all exponential decay sequences $a_n = 2^{-n^{p}}$ where $p\ge1$. We will prove the theorem by exhibiting a subset  of ${\mathcal E}$ that has Hausdorff dimension 1 but Lebesgue measure zero (see Proposition \ref{O-measure-zero}). 
\smallskip
\item Finding a compact set $K$ with $m(\mathcal E_K) > 0$ would prove the non-universality of $A$ satisfying \eqref{moderate_decay_condition}. This is because the set
%To see this, consider the set ${\mathcal E}$ which has a positive Lebesgue measure. Then 
${\mathcal E}_K$ avoids all nontrivial affine copies  of $A\cup \{0\}$ by definition. This would show that $A \cup \{0\}$ is not universal. By a result of Svetic \cite[Lemma 2.1]{S2000}, $A$ would not be universal either. 
\smallskip
\item For certain sequences $\{a_k : k \geq 1\}$ with faster decay than \eqref{moderate_decay_condition}  e.g. $a_k = 2^{-2^k}$, our method provides partial information; specifically, the construction of $K$ leads to a lower bound on the Hausdorff dimension of $\mathcal E_K$ that is positive but strictly smaller than 1.  
%However, this construction does not provide much information  if $a_n$ has a super-exponential decay.
\smallskip
%\item It is natural to ask if  we can construct a set such that ${\mathcal E}$ is of positive Lebesgue measure.  However,  this would imply immediately that $A=\{a_n\}$ is not universal. 
\end{enumerate}
%\end{rem}
\smallskip

The Erd\H{o}s similarity conjecture has long been a focal point of research. We refer the interested reader to \cite{S2000} for a comprehensive survey of the conjecture.  Let us summarize some significant progress here. Given an infinite set $A \subseteq \mathbb R$, Komj\'ath \cite{K1983} proved the existence of a set $E \subseteq \mathbb R$ of positive Lebesgue measure that does not contain any {\em{translate}} of $A$. This result leaves open the possibility that $E$ might contain a {\em{scaled}} copy of $A$. Falconer \cite{F1984} proved non-universality of slowly decaying sequences. Specifically, he showed that sets that contain an infinite sequence $\{x_n : n \in \mathbb N \}$ with $x_{n+1}/x_n \rightarrow 1$ is non-universal.  Bourgain \cite{B1987}  showed that for any three infinite sets $\{S_i : i=1,2,3\}$ in $\mathbb R$, the sumset $S_1+S_2+S_3$ cannot be universal (see also a recent survey by Tao \cite{T2021} about this result). Bourgain remarked that variants of his method can be used to establish non-universality of certain double sums as well, such as $\{2^{-n} \} + \{2^{-n}\}$.  Using a probabilistic construction, Kolountzakis \cite{Kol1997} showed that for any infinite set $A$, one can find a set $E \subseteq [0,1]$ with Lebesgue measure arbitrarily close to 1 such that the exceptional set of dilates 
\[ \{\delta \in \mathbb R : \exists \, x \in \mathbb R \text{ such that } x + \delta A \subseteq E\} \] 
has Lebesgue measure zero. His work also established non-universality of certain infinite structures with large gaps, for instance sets of the form $\{ 2^{-n^{\alpha}}\} + \{2^{-n^{\alpha}} \}$, where $0 < \alpha < 2$. Despite all the efforts above,  determining whether the set $\{2^{-n}: n\in{\mathbb N}\}$ is universal is still an open problem. Our Theorem \ref{mainthm} and Corollary \ref{mainthm-2} show that it is possible to construct a Borel set avoiding all non-trivial affine copies of $\{2^{-n}\}\cup\{0\}$, but which is large in the sense of Hausdorff dimension.

\medskip

The conjecture has also led to many related questions of interest concerning existence or avoidance of patterns in sets. For instance,  Steinhaus's theorem fails for Lebesgue-null sets of large Hausdorff dimension;  given any countable collection of three tuples of points, Keleti \cite{Ke2008} constructed a compact set in $\mathbb R$ of Hausdorff dimension one not containing any nontrivial affine copy of any of the given triplets. In contrast, {\L}aba and Pramanik \cite{LP2009} obtained certain sufficient conditions, involving ball growth and Fourier decay of measures, under which a set of dimension strictly less than one contains a three-term non-trivial arithmetic progression. See \cite{L2017,S2017} for subsequent refined investigation between Fourier dimension and existence of configurations. 

 \medskip

On the other extreme, and though apparently a contradiction in terms, small sets can also contain many patterns. Erd\H{o}s and Kakutani \cite{EK1957} constructed a perfect set of  measure zero but Hausdorff dimension one which contains an affine copy of all finite sets. Recently, M\'ath\'e \cite{M2011} constructed such a perfect set with Hausdorff dimension zero. Molter and Yavicoli \cite{MY2018} constructed an $F_{\sigma}$-set of Hausdorff dimension zero containing affine copies of large families of infinite sets. Yang \cite{Y2019} studied the topological properties of sets containing affine copies of many infinite sequences.

 \medskip

 We now briefly describe the strategy of our proof. Given a fast decaying sequence obeying certain decay conditions, we will describe in Section \ref{Sec-nonuniv-decay}  the construction of a Cantor set $K$ of positive Lebesgue measure that permits an explicit description of $\mathcal O$, a subset of its Erd\H{o}s points. The relevant statement is given in Theorem \ref{mainthm-1}, and it proof appears in Section \ref{Sec-mainthm1-proof}. 
 %In Section \ref{Sec-mainthm2-proof}, we will extract a subsequence from an arbitrary infinite set that satisfies the decay specifications of Section \ref{Sec-nonuniv-decay}; hence, the construction and results of Section \ref{Sec-nonuniv-decay} apply to this subsequence. 
 In Section \ref{Sec-mainthm2-proof}, we estimate the Hausdorff dimension of $\mathcal O$. The proof of Theorem \ref{mainthm} is completed here. 

 \medskip

\section{Setup of the construction} \label{Sec-nonuniv-decay} 

Let us set up the notation used in this paper. The Lebesgue measure of a measurable set $K \subseteq \mathbb R$ will be denoted by $m(K)$. The notation $a_k\searrow0$ (or $a_k\nearrow\infty$) will mean that the sequence $\{a_k: k \geq 1\}$ is strictly decreasing to zero (or strictly increasing to infinity). The notation $a_n\asymp b_n$ means that there exist  absolute constants $C,c>0$ such that  $Cb_n\ge a_n\ge cb_n$ for all sufficiently large $n$.

\medskip

\subsection{A Cantor construction} \label{section-Cantor} Let $N_1,N_2,....$ be a sequence of positive integers greater than $3$. For each $n\in{\mathbb N}$, let us choose a subset  \[ B_n\subset{\mathbb Z}_{N_n}, \quad \text{ where } \quad {\mathbb Z}_{N} := \{0,1,...,N-1\}. \]   
The set $B_n$ will represent the $n^{\text{th}}$ {\em{digit set}} of our Cantor construction. 
\vskip0.1in
\noindent Given the sequence of tuples ${\mathcal N}: = \{(N_n,B_n): n\in{\mathbb N}\}$, we define 
\begin{equation*}
\delta_n := \frac{1}{N_1\cdots N_n}, \quad \text{ and } \quad \Sigma_n := B_1\times...\times B_n = \bigl\{(b_1, \cdots,  b_n): b_j\in B_j \ \forall j=1,\ldots,n\bigr\}.
\end{equation*} 
Each ordered list of integers ${\bf b} = (b_1,\cdots, b_n) \in \mathbb Z_{N_1} \times \cdots \times \mathbb Z_{N_n}$ corresponds to a closed interval $I_{\mathbf b} \subseteq [0,1]$ given by 
\[ I_{{\bf b}} = \sum_{j=1}^nb_j\delta_j+\left[ 0,\delta_n\right]. \]
Among them, the intervals  of the form $\{I_{\mathbf b} : \mathbf b \in \Sigma_n \}$ are called the $n^{\text{th}}$-level {\em{basic intervals}} of the Cantor construction associated to $\mathcal N$. Their union leads to the set
\begin{equation} \label{def-Kn}  
K_n := \bigcup_{\mathbf b \in \Sigma_n} I_{\mathbf b}, 
\end{equation}  often called the $n^{\text{th}}$ {\em{Cantor iterate}}. The iterates $\{K_n : n \in \mathbb N \}$ form a nested sequence of sets that is decreasing in $n$.  Taking their intersection over all the levels $n$, we arrive at the following set generated by the set of tuples ${\mathcal N}$: 
$$
K = K({\mathcal N}) :=  \bigcap_{n=1}^{\infty} K_n = \bigcap_{n=1}^{\infty}\bigcup_{{\bf b}\in\Sigma_n} I_{\bf b}.
$$
A set $K$ obtained through the prescription above is sometimes called a {\it Cantor-Moran set} (as it was first studied by Moran \cite{Moran}). Such a set should be viewed as a natural generalization of the standard middle-third Cantor set in which $N_n = 3$ and $B_n = \{0,2\}$ for all $n$. Similar to the middle-third Cantor set, it is readily seen that elements of $K$ are identified by their digit expansion:
\begin{equation}\label{def-K}
K = \Bigl\{ \sum_{n=1}^{\infty} b_n \delta_n : b_n \in B_n \Bigr\}.  
\end{equation}
The above construction is quite general. For our choice of $K$ we will fix a positive integer $M_n<N_n$ and  choose our digit sets $B_n$ to be 
\begin{equation}\label{eq_B}
B_{n} := {\mathbb Z}_{N_n}\setminus\left(\{1,2,...,M_n\} \cup \{ N_n-1-M_n,...,N_n-2\}\right).
\end{equation}
In other words, $\{ 1, 2, \cdots, M_n \}$ and  $\{ N_n-1-M_n,\cdots,N_n-2\}$ are two forbidden bands for the $n^{\text{th}}$ digit set.  A consequence of this is the following. Suppose that $y = \sum_{j=1}^{m} b_j \delta_j \in K$. Then 
\begin{align} 
b_m = 0 &\text{ implies } \Bigl[ y + (\delta_m, (M_m + 1)\delta_m]  \Bigr] \cap K = \emptyset, \label{forbidden-band} \\ 
b_m = N_m-1 &\text{ implies } \Bigl[ y + [-(M_m + 1)\delta_m, 0)  \Bigr] \cap K = \emptyset. \nonumber 
\end{align} 
We note that this Cantor set satisfies 
\begin{equation} \label{K-symmetry} 
K = 1-K.
\end{equation} 
Indeed, the relation \eqref{K-symmetry} follows from the identity $1 = \sum_{j=1}^{\infty} (N_j-1)\delta_j$; we observe that 
\begin{equation}\label{1-x}
\text{ for all } x = \sum_{j=1}^{\infty} b_j\delta_j,  \quad 1-x = \sum_{j=1}^{\infty} (N_j-1-b_j)\delta_j.
\end{equation}

\medskip

\subsection{A fast decaying sequence} Suppose that $A = \{a_k : k \geq 1\}$ is a sequence of positive numbers such that 
\begin{equation} \label{ak-assumption}
a_1=1, \quad a_k \searrow 0 \quad \text{ and } \quad \frac{a_k}{a_{k+1}} \nearrow \infty \quad \text{ as } k \rightarrow \infty. 
\end{equation} 
For $n \geq 1$, and given positive integers $M_n \geq 1$ and $N_n \geq 3$, we set 
\begin{equation}  
 k_n \coloneqq \sup \Bigl\{ k \geq 1 : \frac{a_k}{a_{k+1}} \leq M_n \Bigr\}. \label{def-kn}
\end{equation} 
Fix $\epsilon > 0$. The main assumptions on $M_n$ and $N_n$ are the following: 
\begin{align}
%&M_n \delta_n \rightarrow 0, \label{Mndn} \\ 
& \frac{\delta_n}{a_{k_n+1}} \rightarrow \infty \text{ as } n \rightarrow \infty, \label{d_na_kn} \\ 
&  M_n  = \lfloor2\epsilon N_n n^{-2}\rfloor \text{ for all } n \geq 1.  \label{M<N}
\end{align} 
where $\floor*{x}$ denotes the largest integer less than or equal to $x$.%The requirement \eqref{ak-assumption} and the strictly increasing property of $M_n$ shows that $k_n < k_{n+1} < \infty$ for each $n$. 
\vskip0.1in

\begin{theorem} \label{mainthm-1}
Let $A = \{a_k: k \ge 1\}$ be a sequence of positive numbers satisfying (\ref{ak-assumption}). Suppose there exist $\epsilon > 0$ and sequences $M_n, N_n$ satisfying the assumptions \eqref{d_na_kn} and \eqref{M<N}. For these $M_n,N_n$ and the Cantor set $K$ in (\ref{def-K}) with digits $B_n$ in (\ref{eq_B}),  the following conclusions hold.
\vskip0.1in 
\begin{enumerate}[(a)]
\item $m(K) > 1 - \pi^2 \epsilon/3$; in other words,  the Lebesgue measure of $K$ can be made arbitrarily close to one by choosing $\epsilon$ sufficiently small.
\vskip0.1in 
\item Define
\begin{equation} \label{def-O} 
{\mathcal O} = \mathcal O[K] \coloneqq \left\{x\in K: x = \sum_{j=1}^{\infty} b_j\delta_j \; \Bigg| \; \begin{aligned} b_j &= 0 \text{ for infinitely many indices } j, \\
b_j &= N_j-1 \text{ for infinitely many indices } j \end{aligned}  \right\}. 
\end{equation} 
Then ${\mathcal O}\subset {\mathcal E}$, where ${\mathcal E} = \mathcal E_K$ is the set of Erd\H{o}s points of $K$ given by \eqref{Eset}.
\end{enumerate} 
\end{theorem}

This theorem is proved in the next section.

\section{Proof of Theorem \ref{mainthm-1}} \label{Sec-mainthm1-proof} 

The proof is divided into several lemmas. The first lemma estimates the Lebesgue measure of the Cantor set. 
%This is standard and we will therefore omit the proof.

\medskip

\begin{lemma} \label{size-lemma}
For the set $K$ defined in \eqref{def-K}, the following conclusions hold. 
\begin{enumerate}[(a)] 
\item The Lebesgue measure of $[0,1] \setminus K$ is 
\begin{equation} \label{size-complement-K}   m \bigl( [0,1] \setminus K \bigr) = \frac{2M_1}{N_1} + (N_1-2M_1)\frac{2M_2}{N_1 N_2} + (N_1-2M_1) (N_2 - 2M_2) \frac{2M_3}{N_1 N_2 N_3} + \cdots.  \end{equation} 
\item Fix $\epsilon > 0$. Suppose that \eqref{M<N} holds.  
Then \begin{equation}  m \bigl( [0,1] \setminus K \bigr) \leq \sum_{j=1}^{\infty}  \frac{2M_j}{N_j} \leq \sum_{n=1}^{\infty} \frac{2\epsilon}{n^2} = \frac{\pi^2 \epsilon}{3}. \label{size-complement-K-2} \end{equation} 
Thus we can make $m(K)$ arbitrarily close to 1 by choosing $\epsilon$ small, provided we can find sequences $M_n$ and $N_n$ that obey \eqref{M<N} for that choice of $\epsilon$.   
\end{enumerate} 
\end{lemma}  
\begin{proof} 
The set $K^c = [0,1] \setminus K$ is an increasing union of the sets $K_n^c = [0,1] \setminus K_n$, where $K_n$ is as in \eqref{def-Kn}. This means that \[ m(K^c) = m(K_1^c) + \sum_{j=1}^{\infty} m(K_{j+1}^c \setminus K_j^c) = m(K_1^c) + \sum_{j=1}^{\infty} m(K_{j+1}^c \cap K_j). \] According to our construction, 
\[ m(K_1^c) = \frac{2M_1}{N_1} \quad \text{ and } \quad m(K_{j+1}^c \cap K_j) = \Bigl[ \prod_{\ell=1}^{j} \Bigl( 1 - \frac{2M_\ell}{N_\ell}\Bigr) \Bigr] \frac{2M_{j+1}}{N_{j+1}} \text{ for } j \geq 1,\]
from which \eqref{size-complement-K} follows.  The $j^{\text{th}}$ summand in \eqref{size-complement-K} is bounded above by $M_j/N_j$, which combined with \eqref{M<N}  leads to the conclusion (\ref{size-complement-K-2}).  
\end{proof} 
\medskip 

Given $A = \{a_k\}$, our next lemma provides a strategy to localize $\delta a_k$; a technique we take advantage of in the sequel to avoid affine copies of $A$.

\begin{lemma} \label{main-lemma}
Suppose that the assumption \eqref{d_na_kn} holds.  Then for all $\delta > 0$ and all sufficiently large $n$, there exists a positive integer $k \leq k_n$ such that 
\[ \delta a_k \in [\delta_n, M_n \delta_n). \] 
\end{lemma} 
\begin{proof} 
Fix $\delta > 0$. By assumption \eqref{d_na_kn} and since $\delta_n \rightarrow 0$, 
%there is a subsequence of the intervals  $(M_n \delta_n, \delta_n/a_{k_n+1})$ that is nested and increasing; hence  
%\[ (0, \infty) = \bigcup_{n=1}^{\infty} \bigl( M_n \delta_n, \frac{\delta_n}{a_{k_n+1}}\bigr).  \]
there exists an integer $n_0 = n_0(\delta)$ such that for all $n \geq n_0$, we have $\delta \in  \bigl(\delta_n, \frac{\delta_n}{a_{k_n+1}}\bigr)$.     
Recalling from \eqref{ak-assumption} that $a_1=1$, we cover the latter interval by disjoint subintervals, as follows: 
\[ \Bigl(\delta_n, \frac{\delta_n}{a_{k_n+1}}\Bigr) \subseteq \bigcup_{k=1}^{k_n} \Bigl[ \frac{\delta_n}{a_k}, \frac{\delta_n}{a_{k+1}}\Bigr).  \]
For the constant $\delta$ under consideration and for every $n \geq n_0(\delta)$, we can therefore find a unique positive integer $k \leq k_n$ such that 
\[ \delta \in \Bigl[ \frac{\delta_n}{a_k}, \frac{\delta_n}{a_{k+1}}\Bigr); \text{ in other words, } \delta a_k \in \Bigl[\delta_n, \frac{a_k}{a_{k+1}}  \delta_n \Bigr). \] It follows from the definition \eqref{def-kn} of $k_n$ that $a_k/a_{k+1} \leq M_n$. Hence $\delta a_k \in [\delta_n, M_n \delta_n)$, as claimed. 
\end{proof}  
%\begin{proposition} \label{mainprop}
%For all $\delta > 0$ and for all $x \in {\mathcal O}$, we can find some positive integer $k$ such that $x + \delta a_k \not\in K$. 
%\end{proposition} 

We conclude this section with the proof of Theorem 2.1, using the two lemmas we just established.
\begin{proof}[{\bf{Proof of Theorem \ref{mainthm-1}}}] 
%Fix $\delta > 0$. We consider two cases, depending on the digit expansion of $x$.
%\vskip0.1in
%\noindent {\bf{Case 1:}} Suppose that $x$ is a left endpoint of one of the basic intervals of $K$, i.e., 
%\begin{equation}  \label{left-endpoint} 
%x = \sum_{j=1}^{n} b_j \delta_j \text{ for some } n \geq 1, \text{ with } b_j \in B_j \text{ for all } 1 \leq j \leq n. 
%\end{equation} 
%As noted in Section \ref{section-Cantor}, $x \in K$, since $0 \in B_n$ for every $n \geq 1$. From the construction described there, it follows that for all $m \geq n+1$, 
%\begin{equation} 
% \bigl( x + [0, \delta_m] \bigr) \cap K \neq \emptyset, \qquad \bigl[x + (\delta_m, (M_m + 1) \delta_m) \bigr] \cap K = \emptyset. \label{set-relations} 
% \end{equation}  
%We first choose $m \geq n+1$ so that the conclusion of Lemma \ref{main-lemma} holds for this choice of $m$ and $\delta$. In other words, we can find $k \leq k_m$ such that $\delta a_k \in [\delta_m, M_m \delta_m)$. Two subcases now arise. If $\delta a_k \in (\delta_m, M_m \delta_m)$, then $x + \delta a_k \in x + (\delta_m, M_m \delta_m)$, hence by the second relation in \eqref{set-relations}, $x + \delta a_k \not\in K$.  
It is clear from Lemma \ref{size-lemma} that  the set $K$ has Lebesgue measure arbitrarily close to 1.  This establishes part (a)  of Theorem \ref{mainthm-1}. 
\vskip0.1in
\noindent It remains to prove part (b), i.e., every $x \in\mathcal O$ is an Erd\H{o}s point. Equivalently, for every $\delta \ne 0$, we aim to establish that $x + \delta A \not\subseteq K$. Let us write $x \in {\mathcal O}$ in terms of its digit expansion 
\[x = \sum_{j=1}^{\infty} b_j \delta_j.  \]
Suppose first $\delta > 0$. It follows from the definition of $\mathcal O$ that $b_n = 0$ for infinitely many indices $n$.   Let us choose $n_0$ for which the conclusion of Lemma \ref{main-lemma} holds for all $n \geq n_0$, and then pick a large enough $m \geq n_0$ so that $b_{m} = 0$. Lemma \ref{main-lemma} ensures the existence of $k \leq k_m$ such that $\delta a_k \in [\delta_m, M_m \delta_m)$.  We can then  write 
\[
x = \sum_{j=1}^{m} b_j\delta_j  + \sum_{j={m+1}}^{\infty} b_j\delta_j : = y+\epsilon_m, \quad \text{ where } 0 \leq \epsilon_m \leq \delta_m. \]
Since $b_m = 0$,  it follows from \eqref{forbidden-band} that $\Bigl[y+ (\delta_m,(M_m+1)\delta_m \bigr] \Bigr]\cap K  = \emptyset$. We consider two cases: 
\begin{itemize} 
\item If $\epsilon_m >0$, this implies that $x + \delta a_k \in y + \epsilon_m + [\delta_m, M_m \delta_m) \subseteq y + (\delta_m, (M_m+1) \delta_m]$.  Hence $x + \delta a_k \not\in K$ by \eqref{forbidden-band}.  

\medskip

\item If $\epsilon_m = 0$, then $x= y$ and $x + \delta a_k\in y+[\delta_m, M_m \delta_m)$. If $x+\delta a_k\ne y+\delta_m$, the point $x+\delta a_k\not\in K$, again by \eqref{forbidden-band}.  This leaves the subcase $\delta a_k = \delta_m$. For this we consider the index $k-1 < k_m$, for which \eqref{def-kn} yields 
\[ 1 < \frac{a_{k-1}}{a_k} < M_m,  \quad \text{hence } \quad x + \delta a_{k-1} = x + \delta_m \frac{a_{k-1}}{a_k} \in x + (\delta_m, M_m \delta_m) \notin K. \]
\end{itemize}  
Combining the two cases, it follows that $x + \delta A \not\subseteq K$ for every $\delta>0$.

\medskip

It remains to investigate the situation where $\delta<0$. In this case, we notice that if $x\in{\mathcal O}$, then $1-x\in {\mathcal O}$ due to (\ref{1-x}). From our previous paragraph, we can find $a_k \in A$ such that $(1-x)-\delta a_k\not\in K$, which implies that 
  $$
  x+\delta a_k\not\in 1-K.
  $$
But $K = 1-K$ by \eqref{K-symmetry}. This obtains the desired conclusion for $\delta<0$, completing the proof.
\end{proof} 

%
%The proof is by contradiction. Suppose there exists $\delta \ne 0$ such that $x + \delta a_k \in K_0$ for all $k \geq 1$. Since $K_0$ is closed, this means that $x \in K_0$.
%\vskip0.1in
%\noindent If $\delta > 0$ and $x \in K_0$, then $x \in K$. By Proposition  \ref{mainprop}, we can find $k \geq 1$ such that $x + \delta a_k \not\in K$, hence $x + \delta a_k \not\in K_0$, a contradiction. 
%\vskip0.1in
%\noindent If $\delta < 0$ and $x \in K_0$, then $x \in 1-K$, i.e., $1-x \in K$. By Proposition \ref{mainprop} again, there exists $k \geq 1$ such that $1 - x + (-\delta) a_k \not\in K$, or equivalently, $x + \delta a_k \not\in 1-K$, hence $x + \delta a_k \not\in K_0$, again a contradiction.   Hence, $K_0$ does not contain any affine copy of the sequence $\{a_k: k \geq 1\}$.
%\vskip0.1in
%\noindent Finally,  the statement concerning the measure of $K_0$ follows from Lemma \ref{size-lemma}. If $\epsilon > 0$ is sufficiently small, then it follows from \eqref{size-complement-K-2} that $m(K) = m(1- K) > 1 - \frac{\pi^2 \epsilon}{6}$. Using the fact that $K\cup(1-K)\subset[0,1]$, we have that 
%$$
%m(K_0) = m(K)+m(1-K)-m(K\cup(1-K)) \ge 2m(K)-1 > 1 - \frac{\pi^2 \epsilon}{3} > 0.
%$$
%This completes the proof of Theorem \ref{mainthm-1}.  
%\end{proof} 

\section{Erd\H{o}s points of Cantor-like sets with forbidden digits} \label{Sec-mainthm2-proof}

\subsection{Uncountability of Erd\H{o}s points} 
We  now turn our attention to proving Theorem \ref{mainthm}.  Let us start by showing that for {\em{any}} convergent sequence $A$ (not necessarily obeying \eqref{moderate_decay_condition}), the construction in Section \ref{Sec-nonuniv-decay} leads to a Cantor-like set $K$ whose set of Erd\H{o}s points is uncountable.  
%Indeed, without considering the Hausdorff dimension, we have the following result.

\begin{theorem} \label{thm_0}
Let $A$ be any null sequence. Then it is possible to choose a null subsequence $\{a_k : k \geq 1\} \subseteq A$ and parameters $M_n$, $N_n$ such that conditions \eqref{d_na_kn} and \eqref{M<N} hold. 
%Given any infinite set $A = \{a_k\}$ such that $a_k>0$ and decreases to $0$. 
As a result, it follows from Theorem \ref{mainthm-1} that there exists a Cantor set $K$ of Lebesgue measure arbitrarily close to 1 whose set of Erd\H{o}s points contains ${\mathcal O}$ and is uncountable. 
\end{theorem} 

\begin{proof} 
We first note that $\IO$ is uncountable by a standard diagonal argument. We only need to see that all points in ${\mathcal O}$ are Erd\H{o}s points.  Given a positive sequence decaying to 0, we can extract a fast-decaying subsequence, which we still denote as $\{a_k : k \geq 1\}$, consisting of positive numbers, such that $a_1 = 1$, 
\begin{equation} 
 a_k \searrow 0,  \quad R_k:=\frac{a_k}{a_{k+1}} \nearrow\infty, \quad R_k > k, \quad R_{k+1} > R_k+1. 
 \label{def-Rk} \end{equation}
Indeed any sequence of positive real numbers decreasing to zero admits a subsequence with the above properties.   Therefore, our proof will follow from Theorem \ref{mainthm-1}  if we can show the existence of the Cantor set satisfying \eqref{d_na_kn} and \eqref{M<N}. 

%\vskip0.1in
%We will shortly specify sequences $\{M_n : n \geq 1\}$ and $\{N_n : n \geq 1\}$ that satisfy the requirements  \eqref{d_na_kn} and \eqref{M<N}. Once this is accomplished, Theorem \ref{mainthm-1} applies, so the set $K_0=E$ provided by this theorem does not contain any affine copy of $\{a_k: k \geq 1\}$ and hence contains no affine copy of $A$. As shown in Lemma \ref{size-lemma}, the set $K_0$ is also of positive Lebesgue measure for small $\epsilon > 0$. 
 Fix $\epsilon>0$ that is the reciprocal of a positive integer and let $C$ be a positive integer to be determined later.  It remains to define the positive integers $M_n$ and $N_n$. Set $M_n := \floor*{R_{C n}}+1$. We also define 
%$r_1=1$, $M_0=10$. Given $r_1, \cdots, r_n$ and $M_1, \cdots, M_{n-1}$ for $n \geq 1$, we choose $r_{n+1} > r_n$ such that 
%\[ \frac{\alpha_{r_n}}{\alpha_{r_{n+1}}} > M_{n-1},\] 
%and set $M_{n}$ to be the unique integer obeying 
%\begin{equation} \label{def-Mn}  \frac{\alpha_{r_n}}{\alpha_{r_{n+1}}} \leq M_{n} < \frac{\alpha_{r_n}}{\alpha_{r_{n+1}}}+1. \end{equation}   
%We then set $a_k = \alpha_{r_k}$, and
\begin{align}  \label{def-Nn}  N_n &:= \frac{2n^2}{\epsilon} M_n,  \;  \text{ so that } \delta_n = \frac1{N_1 \cdots N_n }= \Bigl( \frac{\epsilon}{2}\Bigr)^n\frac{1} {(n!)^2 M_1 \cdots M_n}.
% ; \text{ hence } \\ 4^{-n} &\leq \delta_n \epsilon^{-n} (n!)^2 \prod_{j=1}^{n} R_{j^2} \leq 2^{-n}. \label{dn-bound}
 \end{align}  
Since $1/\epsilon$ is a positive integer, so is $N_n$. The definition \eqref{def-Nn}  immediately implies \eqref{M<N}. 
%In order to establish \eqref{Mndn}, 
%we first observe that the sequence $\{ M_n \delta_n : n \geq 1\}$ is decreasing,  since 
%\[  \frac{M_{n+1} \delta_{n+1}}{M_n \delta_n} = \frac{M_{n+1}}{M_n N_{n+1}} \leq \frac{1}{M_n} < 1.  \] Then 
%we compute from \eqref{def-Nn} that  
%\[ M_n \delta_n = M_n \Bigl(\frac{\epsilon}{2} \Bigr)^n \frac{1}{(n!)^2} \frac{1}{M_1 M_2 \cdots M_n} \leq  \Bigl(\frac{\epsilon}{2} \Bigr)^n \frac{1}{(n!)^2} \rightarrow 0. \]  
To verify \eqref{d_na_kn}, we first recall from \eqref{def-Rk} the requirement that $R_{k+1} > R_k+1$, which implies that 
\begin{equation}\label{M_n>R_n}
R_{C n} < M_n = \floor{R_{Cn}} + 1 \leq R_{Cn} +1 < R_{C n+1}. 
\end{equation}
Hence it follows from the definition of $R_k$ in \eqref{def-kn} that $k_n = \sup \{k : R_k \leq M_n \} =C n$. The definition \eqref{def-Rk} and $a_1=1$ implies  the relation $1/a_{k+1} = R_1 \cdots R_k$ for all $k \geq 1$. Combining this with \eqref{def-Nn}, we obtain:  
$$
\begin{aligned}
 \frac{\delta_n}{a_{k_n+1}} = \frac{\delta_n}{a_{C n+1}} &= \delta_n 
\cdot\prod_{j=1}^{C n} R_j = \Bigl(\frac{\epsilon}{2} \Bigr)^{n} \frac{\prod_{j=1}^{Cn}R_j}{(n!)^2 M_1 \cdots M_n} \\
 &\ge  \Bigl( \frac{\epsilon}{2}\Bigr)^n \frac{\prod_{j=1}^{C n} R_j}{(n!)^2  \prod_{j=1}^n (R_{C j}+1)}  \\
 & \ge \Bigl( \frac{\epsilon}{4}\Bigr)^n \frac{\prod_{j=1}^{C n} R_j}{(n!)^2  \prod_{j=1}^n R_{C j}} \   \ \ (\mbox{using $R_{Cj}+1<2R_{Cj}$})\\  
&  = \Bigl( \frac{\epsilon}{4}\Bigr)^n \frac{1}{(n!)^2} \cdot \prod_{1\le j\le C n, j\not\in C {\mathbb Z}}R_j.
\end{aligned}
$$
%\geq \Bigl(\frac{\epsilon}{4} \Bigr)^n \frac{1}{(n!)^2} \prod_{j=1}^{n} R_{j^2}^{-1} \cdot \prod_{j=1}^{n^2} R_j =  \Bigl(\frac{\epsilon}{4} \Bigr)^n \frac{1}{(n!)^2} \prod^{\ast}_{1 \leq j \leq n^2} R_j, 
%\end{align*}
The assumed bound $R_k > k$ lets us estimate the last quantity from below:
\begin{equation}  \frac{\delta_n}{a_{k_n+1}} \geq \Bigl(\frac{\epsilon}{4} \Bigr)^n \frac{1}{(n!)^2} \prod_{1\le j\le C n, j\not\in C {\mathbb Z}} j = \Bigl[ \Bigl(\frac{\epsilon}{4} \Bigr)^n \frac{1}{(n!)^2} \Bigr] \frac{(Cn)!}{C^n \cdot n!} = \Bigl(\frac{\epsilon}{4C} \Bigr)^n \frac{(C n)!}{(n!)^3}, \label{final-lower-bound} \end{equation} 
By Stirling's approximation, $n! \asymp \sqrt{2\pi n} \left(\frac{n}{e}\right)^n$. Using this, we estimate the lower bound in \eqref{final-lower-bound}: 
$$
 \Bigl(\frac{\epsilon}{4C} \Bigr)^n \frac{(C n)!}{(n!)^3} \asymp \Bigl(\frac{\epsilon}{4 C} \Bigr)^n \frac{\sqrt{2\pi C n} (Cn)^{C n} e^{-C n} }{(2\pi n)^{3/2} n^{3n} e^{-3n}}\asymp  \Bigl(\frac{\epsilon C^C}{4 C} \Bigr)^n \frac{n^{(C-3)n}}{n e^{(C-3)n}}.
$$
If we  take for example $C = 5$ and let $\kappa = \frac{20 e^2}{ 5^5 \epsilon}$, the last quantity becomes
$$
\frac{n^{2n}}{n \kappa^{n}} = \frac{n^n}{n} \cdot \frac{n^n}{\kappa^n} 
$$
which  diverges to infinity as $n \rightarrow \infty$.  We have thus verified all the requirements, and hence completed the proof.
\end{proof} 

\subsection{Hausdorff dimension of a subset of Erd\H{o}s points} 

We now need to estimate the Hausdorff dimension of the set $\mathcal E = \mathcal E_K$ consisting of the Erd\H{o}s points of $K$. To do this, we will identify a subset $\mathcal O_S$  of Hausdorff dimension 1 contained in $\mathcal O$. Since $\mathcal O \subseteq \mathcal E$ by Theorem \ref{mainthm-1}(a), we conclude that $\mathcal E$ must have Hausdorff dimension 1 as well. We now specify our desired subset.

\smallskip

Given a sequence of tuples $\mathcal{N} = \{ (N_n, B_n) | n\in \N \}$, we pick a subsequence indexed by $S = \{ n_1<  n_2 < n_3 < n_4 < \dots \} $. With this subset we define the set 
\begin{equation}\label{Os}
\IO_S \coloneqq \left\{ x= \sum_{j=1}^\infty \frac{b_j}{N_1\dots N_j}: b_j \in B_j, b_{n_{2j-1}}=0, b_{n_{2j}} = N_{2j}-1 \right\} \subseteq \IO .
\end{equation}
We can define this set another way,
\[
\IO_S =  \bigcap_{k=1}^\infty E_k 
\]
where 
\begin{align*}
    E_1 &= \bigcup \Bigl\{ I_{\mathbf b} : \mathbf b \in \mathfrak B_1 := \prod_{i=1}^{n_1-1} B_i \times \{0\} \Bigr\},  \\
    E_2 &= \bigcup \Bigl\{ I_{\mathbf b} :  \mathbf b \in \mathfrak B_2 := \mathfrak B_1 \times \prod_{i=1}^{n_2-1} B_i \times \{N_2-1\} =: \mathfrak B_2  \Bigr\} \\
    &\vdots \\
    E_k &= \bigcup \Bigl\{ I_{\mathbf b}:  \mathbf b \in \mathfrak B_k := \mathfrak B_{k-1} \times \prod_{i=1}^{n_k-1} B_i \times \{\varrho_k\}  \Bigr\}, \text{ where } \\ 
    \varrho_k &= \begin{cases}  
    0 &\text{ if $k$ is odd, } \\ 
    N_k-1 &\text{ if $k$ is even.} 
    \end{cases} 
\end{align*}
For $\mathbf b \in \mathfrak B_k$, the Lebesgue measure of the interval $I_{\mathbf b}$ is given by 
\[
m(I_{\mathbf b}) = \delta_{n_k} = \frac{1}{N_1 N_2 \dots N_{n_k -1} N_{n_k}}.
\]
We denote $|B|$ as the cardinality of a finite set $B$. We have the following proposition. 
\begin{proposition}\label{prop_Hdim}
Using the above notation, the Hausdorff dimension of $\IO_S$ is equal to
\[
1-\limsup_{j\rightarrow \infty} \frac{\log ( \prod_{\ell=1}^{j-1} |B_{n_\ell}|)}{\log (\prod_{\ell=1}^{n_{j-1}} N_{\ell})} .
\]
\end{proposition}

\begin{proof}
Using the well-known result in Falconer textbook \cite[Example 4.6, Chapter 4]{Falconer}, a Cantor set $\IO_S$ constructed as in Section \ref{Os} has Hausdorff dimension
\begin{equation}\label{falconerhdim}
\dim_H \IO_S \ge \liminf_{j\rightarrow \infty} \frac{\log (m_1 \dots m_{j-1})}{-\log (m_j \epsilon_j)}, \text{ where }  
\end{equation}
\begin{align*}
m_j &= \text{ number of }j^{th}\text{ level intervals in a } (j-1)^{\text{th}} \text{ level interval } \\ 
&=  |B_{n_{j-1}+1}| \dots |B_{n_{j}-1}|, \text{ and } \\ 
\epsilon_j &= \text{ minimum gap length among } j^{\text{th}} \text{ level intervals} \\
&=   \delta_{n_j - 1} - \delta_{n_j} = \delta_{n_j - 1} \left(1-\frac1{N_{n_j}}\right). 
\end{align*} 
%For $\IO_S$, $m_j$ is given by
%\[
%m_j = |B_{n_{j-1}+1}| \dots |B_{n_{j}-1}| 
%\]
%while $\epsilon_j$ is given by
%\[
%\epsilon_j = \delta_{n_j - 1} - \delta_{n_j} = \delta_{n_j - 1} \left(1-\frac1{N_{n_j}}\right).
%\]
From (\ref{M<N}) we have that 
\begin{equation}\label{est_B_n}
N_n (1-\frac{2\epsilon}{ n^{2}}+\frac1{N_n})\ge |B_n| = N_n-2M_n \ge  N_n (1-\frac{2\epsilon}{ n^{2}}).
\end{equation}
Using this, the fraction in \eqref{falconerhdim} can be written as 
%\begin{align*}
%&=  \frac{\log (\prod_{l=1}^{n_{j-1}-1}|B_l|/\prod_{l=1}^{j-1} |B_{n_l}|) }{-\log (|B_{n_{j-1} + 1}| \dots |B_{{n_j}-1}| (\delta_{n_j - 1} - \delta_{n_j}))}\\ 
%    &=  \frac{\log (\prod_{l=1}^{n_{j-1}-1}|B_l|) - \log (\prod_{l=1}^{j-1} |B_{n_l}|) }{-\log (|B_{n_{j-1} + 1}|\cdots |B_{n_j-1}| \cdot (\delta_{n_j-1} - \delta_{n_j}))}\\
%    &=  \frac{\log (\prod_{l=1}^{n_{j-1}-1}\abs{B_l}) - \log (\prod_{l=1}^{j-1} \abs{B_{n_l}}) }{-\log (\abs{B_{n_{j-1} + 1}} \cdots \abs{B_{n_j-1}} \cdot \delta_{n_j-1} (1 - \frac{\delta_{n_j}}{\delta_{n_j-1}}))}\\
%    &=  \frac{\log (\prod_{l=1}^{n_{j-1}-1}\abs{B_l}) - \log (\prod_{l=1}^{j-1} \abs{B_{n_l}}) }{-\log( \frac{1}{N_1\cdots N_{n_{j-1}}} \frac{\abs{B_{n_{j-1} + 1}}}{N_{n_{j-1}+1}} \cdots \frac{\abs{B_{n_j-1}}}{N_{n_{j-1}}} (1 - \frac{1}{N_{n_j}}) )}\\
%    &=  \frac{\log (\prod_{l=1}^{n_{j-1}-1}\abs{B_l}) - \log (\prod_{l=1}^{j-1} \abs{B_{n_l}}) }{-\log( \frac{1}{N_1\cdots N_{n_{j-1}}}) -\log ( \frac{\abs{B_{n_{j-1} + 1}}}{N_{n_{j-1}+1}} \cdots \frac{|{B_{n_j-1}}|{N_{n_{j-1}}}) -\log  (1 - \frac{1}{N_{n_j}}) }
%\end{align*}
%Up to this point we only used $\log$ properties and algebraic manipulations but, before we continue we want to remind the reader of the following substitution we will perform
%\begin{align*}
%    |B_l| &\approx N_l - M_l \\
%    &\ge N_l - \frac{\epsilon}{l^2} N_l = N_l(1-\frac{\epsilon}{l^2}) .
%\end{align*}
%Now we proceed with our substitutions and manipulations.

\begin{align*}
 \frac{\mathfrak N}{\mathfrak D} &:=  \frac{\log (m_1 \dots m_{j-1})}{-\log (m_j \epsilon_j)} \text{ where } \\
   \mathfrak N &:= \log \bigl(\prod_{\ell=1}^{n_{j-1}}|B_\ell| \bigr) - \log \bigl(\prod_{\ell=1}^{j-1} |B_{n_{\ell}}| \bigr) \\
   &\geq  \log \Bigl(\prod_{\ell=1}^{n_{j-1}} N_{\ell} \Bigr) + \log \Bigl[  \prod_{\ell=1}^{n_{j-1}-1} (1-\frac{2\epsilon}{\ell^2})  \Bigr] - \log \Bigl( \prod_{\ell=1}^{j-1} |B_{n_\ell}| \Bigr), \text{ and } \\ 
   \mathfrak D &= -\log \Bigl( \frac{1}{N_1\cdots N_{n_{j-1}}} \Bigr) -\log \Bigl( \frac{|B_{n_{j-1} + 1}|}{N_{n_{j-1}+1}} \cdots \frac{|B_{n_j-1}|}{N_{n_{j}-1}} \Bigr) -\log  \bigl(1 - \frac{1}{N_{n_j}} \bigr) \\
   &\leq \log \bigl(\prod_{\ell=1}^{n_{j-1}} N_{\ell} \bigr) - \log \Bigl(\prod_{\ell=n_{j-1}}^{n_j -1} (1-\frac{2\epsilon}{\ell^2}+\frac1{N_{{\ell}}}) \Bigr)   - \log \Bigl(1-\frac{1}{N_{n_j}} \Bigr).
   \end{align*} 
   Simplifying the expressions above leads to 
\begin{align*}
\dim_{H}(\mathcal O_S) &= \lim_{j \rightarrow \infty} \frac{\mathfrak N}{\mathfrak D} \geq \lim_{j \rightarrow \infty}\Biggl[1 + \frac{\log [  \prod_{\ell=1}^{n_{j-1}} (1-\frac{2\epsilon}{\ell^2})  ]}{\log (\prod_{\ell=1}^{n_{j-1}} N_\ell)} - \frac{\log \Bigl( \prod_{\ell=1}^{j-1} |B_{n_{\ell}}| \Bigr)}{\log (\prod_{\ell=1}^{n_{j-1}} N_{\ell})}  \Biggr]  \\ & \hskip1.5in \times \Biggl[1 - \frac{\log [\prod_{\ell=n_{j-1}}^{n_j -1} (1-\frac{2\epsilon}{\ell^2}+\frac1{N_{{\ell}}}) ]}{\log (\prod_{\ell=1}^{n_{j-1}} N_{\ell})}   - \frac{\log(1-\frac{1}{N_{n_j}})}{\log (\prod_{\ell=1}^{n_{j-1}} N_{\ell})} \Biggr]^{-1}.
\end{align*}
We note that the products $\prod_{\ell=1}^{\infty} (1-\frac{2\epsilon}{\ell^2}) $ and $\prod_{\ell=1}^{\infty} (1-\frac{2\epsilon}{\ell^2}+\frac1{N_{{\ell}}})$ are finite and positive. Therefore,   except the last term in the numerator, all the other terms in $\mathfrak N$ and $\mathfrak D$ tend to zero as $j$ goes to infinity. We obtain that 
\[
\mbox{dim}_H({\mathcal O}_{\mathcal S})\ge 1-\limsup_{j\rightarrow \infty} \frac{\log ( \prod_{\ell=1}^{j-1} |B_{n_{\ell}}|)}{\log (\prod_{\ell=1}^{n_{j-1}} N_{\ell})}. 
\]

In remains to establish that $\dim_H(\mathcal O_S)$ is equal to the right hand side above. It follows from \cite[Example 4.6]{Falconer} that this is a consequence of the condition $m_j\epsilon_j\ge c \delta_{n_{j-1}}$. We will verify this condition. Applying the right inequality in \eqref{est_B_n}, we obtain 
\begin{align*}
m_j\epsilon_j &=   |B_{n_{j-1}+1}| \dots |B_{n_{j}-1}| \delta_{n_j - 1} \left(1-\frac1{N_{n_j}}\right) \\
&\ge  \delta_{n_{j-1}}\times \Biggl[\prod_{\ell={n_{j-1}+1}}^{n_{j}-1}(1-\frac{2\epsilon}{ \ell^{2}}) \Biggr]\left(1-\frac1{N_{n_j}}\right) \\ 
&\geq \frac{1}{2} \Biggl[\prod_{\ell=1}^{\infty}(1-\frac{2\epsilon}{ \ell^{2}}) \Biggr] \delta_{n_{j-1}} \geq c \delta_{n_{j-1}}.
\end{align*} 
The last step follows from the fact that the product is bounded below by a positive number. 
%$\prod_{\ell={1}}^{\infty}(1-\frac{2\epsilon}{ \ell^{2}} )$,  $m_j\epsilon_j\ge c \delta_{n_{j-1}}$ 
%is satisfied and  the equality of Hausdorff dimension  holds. 
This completes the verification of the condition, and hence the proof.
\end{proof}

\subsection{Proof of Theorem \ref{mainthm}} 
We are now in a position to compute the Hausdorff dimension of the set of Erd\H{o}s points of $K$, and complete the proof of Theorem \ref{mainthm}. 

\begin{proof}
Let us recall that $R_k \coloneqq \frac{a_k}{a_{k+1}}$.  Set
$$
M_n := \floor*{R_{Cn}}+1, \quad N_n = \frac{2n^2}{\epsilon} M_n, \quad B_n \text{ as in \eqref{eq_B}}, \quad \mathcal N = \bigl\{(N_n, B_n): n \in \mathbb N \bigr\},   
$$ 
and construct the Cantor set $K = K(\mathcal N)$ as described in Section \ref{section-Cantor}. 
In Theorem \ref{thm_0}, we deduced that ${\mathcal O} = \mathcal O[K]\subset {\mathcal E}$ where any integer $C\ge 5$ will work. 
%Moreover,  $K(\mathcal{N})$ was defined by $M_n$ and $N_n$ in Section 2 and
%$$
%M_n := \floor*{R_{Cn}}+1, \ N_n = \frac{2n^2}{\epsilon} M_n.
%$$
To complete the proof, we need to show $\mbox{dim}_H({\mathcal O}) = 1$. Using Proposition \ref{prop_Hdim}, it suffices to show for some subsequence ${\mathcal S}$, $\mbox{dim}_H({\mathcal O}_{\mathcal S}) = 1$. In other words, for $\IO_\mathcal{S}$ we need to establish that
$$
\limsup_{j\rightarrow \infty} \frac{\log ( \prod_{\ell=1}^{j-1} |B_{n_{\ell}}|)}{\log (\prod_{\ell=1}^{n_{j-1}} N_{\ell})} = 0.
$$
In view of (\ref{est_B_n}), this is equivalent to showing that  
\begin{equation} \label{Hdim-verify} 
\limsup_{j\rightarrow \infty} \frac{\log ( \prod_{\ell=1}^{j} N_{n_{\ell}})}{\log (\prod_{\ell=1}^{n_{j}} N_{\ell})} = 0.
\end{equation} 
Expressing $N_n$ in terms of $R_n$, we obtain
\begin{align}
\frac{\log ( \prod_{\ell=1}^{j} N_{n_{\ell}})}{\log (\prod_{\ell=1}^{n_{j}} N_{\ell})} \le & \frac{\log ( \prod_{\ell=1}^{j} \frac{2n_{\ell}^2}{\epsilon}  (2R_{Cn_{\ell}}) )}{\log (\prod_{\ell=1}^{n_{j}} \frac{2\ell^2}{\epsilon} R_{C\ell})} \nonumber\\
\le& \frac{j \log (4/\epsilon)+ 2\log (n_1...n_{j})+ \log (\prod_{\ell=1}^{j}R_{n_{\ell}})}{n_{j}\log(2/\epsilon) +2\log (n_{j}!)+\log(\prod_{\ell=1}^{n_{j}}R_{C\ell})} \nonumber\\
\lesssim_{\epsilon} & \frac{\log \bigl(\prod_{\ell=1}^{j}R_{Cn_{\ell}}\bigr)}{\log\bigl(\prod_{\ell=1}^{n_{j}}R_{C\ell}\bigr)} = \Bigl[\sum_{\ell=1}^{j} \exp \bigl(\varphi(Cn_{\ell}) \bigr) \Bigr] \Bigl[\sum_{\ell=1}^{n_j} \exp \bigl( \varphi(C\ell) \bigr) \Bigr]^{-1},
%\frac{j \log (4/\epsilon)}{n_{j-1}\log(2/\epsilon)}+ \frac{\log (n_1...n_{j-1})}{\log (n_{j-1}!)}+ \frac{\log (\prod_{\ell=1}^{j-1}R_{Cn_{\ell}})}{\log(\prod_{\ell=1}^{n_{j-1}}R_{C\ell})}. 
\label{eqest}
 \end{align}
 where the constant implicit in $\lesssim_{\epsilon}$ depends only on $\epsilon$, and the last inequality follows from the assumption \eqref{def-Rk} that $R_{Ck}>R_k > k$, so that 
 \[ j \log(4/\epsilon) + 2 \log(n_1 \cdots n_{j}) \lesssim_{\epsilon} \log \Bigl(\prod_{\ell=1}^{j}R_{Cn_{\ell}} \Bigr). \] 
%By our assumption that $\log R_n = e^{\varphi(n)}$ and $\frac{\varphi(n)}{n}\to 0$, we can write 
%$$
%\log R_{C \ell} = e^{\varrho_{\ell} \cdot  \ell}
%$$
%where $\varrho_{ n} = \frac{\varphi(Cn)}{ n}$. Note that $\varrho_n \to 0 $ when $n\to \infty$. We fix a decaying sequence  $\nu_j\to 0$ and $\nu_{j}<1$. Then we define
%$$
%n_j = \min \{k: \varrho_{ k} < \nu_j\}.
%$$
%By choosing $\nu_j$ decaying sufficiently fast, we can assume $ \frac{j}{n_j}\to 0$ as $j\to\infty$. Moreover, we know that $\varphi(Cn)\to\infty$, so $n\varrho_n\to\infty$ as $n\to\infty$. Hence, for sufficiently large $j$, 
%\begin{equation}\label{eq111}
%\nu_{j} >\varrho_{n_j} \ge \frac{2}{n_j}, \ \mbox{which means} \ \nu_j n_j >2.
%\end{equation}
%As $\frac{j}{n_j} \to0$, the first term in (\ref{eqest}) goes to zero, and the second term(by Stirling approximation),
%$$
%\frac{\log (n_1...n_{j-1})}{\log (n_{j-1}!)} \asymp \frac{\log (n_{j-1}^{j-1})}{\log n_{j-1}^{n_{j-1}}e^{-n_{j-1}} n_{j-1}^{1/2}} \le  \frac{(j-1)\log (n_{j-1})}{n_{j-1}(\log n_{j-1} -1)}\to 0.
%$$
%We now deal with the last term in (\ref{eqest}). 
Let us choose the subsequence $n_{\ell}$ in the following way: $n_{\ell}$ is the largest integer $n$ such that $\varphi(Cn) \leq \ell+1$. From our assumption, $\frac{\varphi (Cn)}{n}\to 0$ as $n\to \infty$.  In Lemma \ref{lemma_equivalent_decay} below, we will see that this condition is equivalent to  
\begin{equation} \label{omega-conclusion}
   \omega(r) := \# \bigl\{\ell : r < \varphi(C\ell) \leq {r+1} \bigr\} \rightarrow \infty \text{ as } r \rightarrow \infty. 
\end{equation} 
We deduce 
%allows us to conclude that $\varphi(n_{\ell}) \geq \ell$. Combining these observations, we deduce 
\begin{align} 
\sum_{\ell=1}^{j} \exp \bigl(\varphi(Cn_\ell) \bigr) &\leq \sum_{\ell=1}^{j} e^{\ell+1} \lesssim e^{j}, \; \text{ whereas } \label{phi-num} \\
\sum_{\ell=1}^{n_j} \exp \bigl( \varphi(C\ell) \bigr)  &\geq \sum_{r=1}^{j} e^{r} \# \bigl\{\ell : r+1 > \varphi(C\ell) \geq r \bigr\}
=  \sum_{r=1}^{j} e^{r} \omega(r) \geq e^{j} \omega(j).  \label{phi-denom} 
\end{align}
Combining \eqref{phi-num} and \eqref{phi-denom} with \eqref{eqest} and \eqref{omega-conclusion}, it follows that 
\[\frac{\log ( \prod_{\ell=1}^{j} N_{n_{\ell}})}{\log (\prod_{\ell=1}^{n_{j}} N_{\ell})} \lesssim \Bigl[\sum_{\ell=1}^{j} \exp \bigl(\varphi(Cn_{\ell}) \bigr) \Bigr] \Bigl[\sum_{\ell=1}^{n_j} \exp \bigl(\varphi(C\ell) \bigr) \Bigr]^{-1} \lesssim \frac{1}{\omega(j)} \rightarrow 0 \text{ as } j \rightarrow \infty, \]
as claimed in \eqref{Hdim-verify}. 
\end{proof}

\begin{lemma} \label{lemma_equivalent_decay}
  Let $x_n$ be a strictly increasing sequence of positive real numbers. Then $\frac{x_n}{n}\to 0$ as $n\to\infty$ if and only if 
$$
\omega(k) = \# \bigl\{\ell : k <  x_{\ell} \leq {k+1} \bigr\} \rightarrow \infty \text{ as } k \rightarrow \infty. 
$$
\end{lemma}

\begin{proof}
($\Longrightarrow$) Suppose towards a contradiction, it is possible that $\sup_{r} \omega(r) =: C_0 < \infty$. For every $n \geq 1$, let $r = r(n)$ be the unique integer such that $x_n \in (r, r+1]$. Since $x_n$ is strictly increasing, we conclude that $r(n)$ is monotone non-decreasing, with $r(n) \to \infty$ as $n \rightarrow \infty$. Thus
\[ n \leq \# \bigl\{\ell : 1< x_{\ell} \leq {r(n)+1}\bigr\} = \sum_{k=1}^{r(n)} \omega(k) \leq C_0r(n). \]  
But this means that 
\[ \frac{x_n}{n} \geq \frac{{r(n)}}{C_0r(n)}= \frac{1}{C_0} > 0,\]
contradicting the assumption that $\varphi(n)/n \rightarrow 0$. 

\medskip

\noindent($\Longleftarrow$) For every $M\ge 1$, there exists $k_M\ge 1$ such that $\omega(k)\ge M$ for all $k\ge k_M$. With the same definition of $r(n)$, we know that 
$$
n\ge \#\bigl\{\ell : 1< x_{\ell} \leq {r(n)}\bigr\} = \sum_{k=1}^{r(n)}\omega (k)
$$
Hence, 
$$
\frac{x_n}{n} \le \frac{r(n)+1}{\sum_{k=1}^{r(n)}\omega (k)} \le \frac{r(n)+1}{\sum_{k=k_M}^{r(n)}\omega (k)} \le \frac{r(n)+1}{(r(n)-k_M)M},
$$
meaning that $\limsup\limits_{n\to\infty}\frac{x_n}{n}\le \frac1{M}$ for all $M\ge 1$. As $M$ is arbitrary, the proof is complete. 

\end{proof}

\subsection{Remark and open questions.} Our main result found a subset ${\mathcal O}$ of Hausdorff dimension 1 inside the set of Erd\H{o}s points of $K$. The following proposition shows however that ${\mathcal O}$ has Lebesgue measure zero.

\begin{proposition} \label{O-measure-zero}
Under the assumption of Theorem \ref{mainthm-1}, the set ${\mathcal O}$ constructed in (\ref{def-O}) has Lebesgue measure zero.
\end{proposition}

\begin{proof}
For $n\in{\mathbb N}$, let 
$$
J_n = \bigcup \Bigl\{ I_{\mathbf b} : \mathbf b \in  \prod_{i=1}^{n-1} B_i \times \{0, N_n-1\} \Bigr\}
$$
This $J_n$ collects all the $n$th level intervals that have $0$ or $N_n-1$ at the $n^{th}$ digit. In particular, all the points in $K$ such that $b_n = 0$ or $ N_n-1$ are in $J_n$. By definition of ${\mathcal O},$
\begin{equation}\label{eq_Osubset}
{\mathcal O} \subset \bigcap_{k=1}^{\infty} \bigcup_{n= k}^{\infty} J_n.
\end{equation}
This implies that 
$$
m(J_n)= \frac{1}{N_1....N_n} \cdot |B_1|....|B_{n-1}|\cdot 2 \le \frac{2}{N_n}.
$$
From (\ref{def-Kn}),  the Lebesgue measure of $K$ is equal to 
$$
\lim_{n\to\infty} m(K_n)= \prod_{n=1}^{\infty} \left(1-\frac{2M_n}{N_n}\right).
$$
This number is positive if and only if $\sum_{n=1}^{\infty} M_n/N_n<\infty$. But $M_n>1$, this implies that 
$$
\sum_{n=1}^{\infty}m(J_n) \le 2\sum_{n=1}^{\infty} \frac{1}{N_n} <\infty.
$$
Hence, by the Borel-Cantelli Lemma and (\ref{eq_Osubset}), $m({\mathcal O})  = 0$.
\end{proof}

We do not know if there are more points in ${\mathcal E}_K$ other than points in ${\mathcal O}$ and we do not know how the decay rate condition in Theorem \ref{mainthm} be removed. In view of this, we conclude the paper with two open problems.

\medskip

\begin{enumerate}[1.]
\item Can condition (\ref{moderate_decay_condition}) about the decay rate of the sequence in Theorem \ref{mainthm} be removed? 
\item Can one strengthen the arguments in this paper to verify whether $\mathcal E_K$ has positive Lebesgue measure? If true, that would resolve the Erd\H{o}s similarity conjecture for sequences $\{a_k\}$ with \eqref{moderate_decay_condition}. 
\end{enumerate}

\appendix

\section{Measurability of Erd\H{o}s Points}
Here we prove the statements (a) and (b) in page \pageref{appendix-results}.  
%We now prove that the set of all Erd\H{o}s points in a closed set must be a Borel set. We first show that it suffices to show the Erd\H{o}s similarity conjecture for closed sets.

\begin{lemma} \label{universal-compact-lemma}
A set $A$ is universal if and only if every compact set $K$ of positive Lebesgue measure contains a nontrivial affine copy of $A$. 
\end{lemma}

\begin{proof}
We only need to prove the ``if'' part of the statement. Suppose that ${\mathcal A}$ is universal for all compact sets. Given a set of positive Lebesgue measure $E$, the inner regularity of Lebesgue measure allows us to find a compact subset $K\subset E$ with positive Lebesgue measure.  By our assumption, $K$, and hence $E$, contains a nontrivial affine copy of $A$. The proof is complete. 
\end{proof}

\medskip

%In view of this lemma, we just need to focus our attention on closed sets.  
Let $K$ be any compact subset of $\R$ and let $A = \{a_{k}\}$ be a bounded sequence of real numbers. Define 
$$
{\mathcal F} = \{(x,\delta)\in\R^2: x+\delta{A}\subset K\} \subseteq \mathbb R^2, 
$$
and 
$$
F=  \{x\in K: \exists \delta\ne 0, x+\delta{ A}\subset K \}.
$$
We further write 
\[ {\mathcal F}  = \bigcup_{j=-\infty}^{\infty}{\mathcal F}_{j}, \quad \text{ where } {\mathcal F}_{j} = \bigl\{(x,\delta)\in\mathcal F: \delta\in [-2^{j+1},-2^j]\cup[2^j,2^{j+1}] \bigr\}. 
\]
We note that $\mathcal{E}= K\setminus F = K\cap F^c$. In order to show that  the set $\mathcal E$ of Erd\H{o}s points is Borel measurable, it suffices to prove the same for  $F$. We do this below. 

\begin{proposition} \label{E-Borel-prop}
For each $j$, the set ${\mathcal F}_{j} $ is closed in $\R^2$ and $F = K \setminus \mathcal E$ is Borel measurable on $\R$.
\end{proposition}

\begin{proof} 
Let $(x_k,\delta_k)\in {\mathcal F}_{j}$ and $(x_k,\delta_k)\to (x,\delta)\in \R^2$. Then for all $n\in{\mathbb N}$, 
$$
x_k+\delta_ka_n\in K.
$$
As $K$ is closed, $x+\delta a_n\in K$ also and $\delta$ is still in the range of interest. Hence, ${\mathcal F}_{j}$ is closed and thus compact since the set is bounded. To see that $F$ is measurable, we note that $F= \pi_x({\mathcal F}) = \bigcup_{j=-\infty}^{\infty} \pi_x ({\mathcal F}_{j})$, where $\pi_x$ is the projection onto the $x$-axis. As ${\mathcal F}_{j}$ is closed, ${\mathcal F}$ is a countable union of compact sets  and $F$ is a countable union of the image of the compact sets under $\pi_x$, which are compact. Hence, $F$ is a $F_{\sigma}$ set. 
\end{proof}
\medskip

 That gives us our desired result that $\mathcal{E}$ is in fact a measurable set. We end this appendix by the following interesting proposition  which says that if $A$ is universal for all sets of positive Lebesgue measure, then almost all points of $K$ are not Erd\H{o}s points. In other words, we should be able to find affine copies almost everywhere. 

\begin{proposition}
Suppose that ${A}$ is universal for all sets of positive Lebesgue measure. Then for all closed set $K$, $m(K) = m(F)$.  
\end{proposition}

\begin{proof}
Suppose that  $A$ is universal.  Using Svetic's paper Lemma 2.1, we know also that  $\overline{A}$ is also universal.    Assume for the sake of contradiction $m(K\setminus F)>0$. Then $K\setminus F$ will also contain an affine copy of $\overline{A}$ since $K\setminus F$ is measurable. However, 
$$
K \setminus F = \{x\in K: \forall\delta\ne0, \exists n \  \mbox{s.t.}   \ x+\delta a_n   \not\in K\}.
$$
This means that for all $x\in K\setminus 
F$ and for all $\delta\ne0$, we can find $n$ such that $x+\delta a_n\not\in K\setminus F$. This means there is no affine copy of $\overline{A}$ in $K\setminus F$. Hence, $\overline{A}$ is not universal, a contradiction. Thus $m(K\setminus F) = 0$ and hence $m(K) = m(F)$.
\end{proof}

%\begin{proposition}
%Suppose that ${\mathcal A}$ is universal for all closed sets of positive Lebesgue measure in $[0,1]$. Then ${\mathcal A}$ is universal for all sets of positive Lebesgue measures. 
%\end{proposition}
%
%
%\begin{proof}
%Let $K$ be any closed sets of positive Lebesgue measure. Let $x$ be a Lebesgue density points of $K$. Then for all $\eta>0$, we can find $\epsilon>0$ such that 
%$$
%m(K\cap [x-\epsilon,x+\epsilon]) \ge (1-\eta) 2\epsilon.
%$$
%We now consider the rescaled set
%$$
%\widetilde{K} = \frac1{2\epsilon} (K\cap [x-\epsilon,x+\epsilon]) + \left(\frac{1}{2}- \frac{x}{2\epsilon}\right).
%$$
%Then $m(\widetilde{K})\ge 1-\eta$ and $\widetilde{K} \subset[0,1]$. By our assumption, we can find an affine copy of ${\mathcal A}$ in $\widetilde{K}$. But rescaling back, we find an affine copy inside $K\cap [x-\epsilon,x+\epsilon]$.
%\end{proof}
%
%This proposition shows us that if we want to show ${\mathcal A}$ is not universal, we just need to show it is not universal for closed sets inside $[0,1]$.
%\medskip

\end{document}